\documentclass[11pt,a4paper]{article}
\usepackage{graphicx}
\usepackage{amsmath,amsfonts,amssymb,amsthm}
\title{The middle hedgehog of a planar convex body}
\author{Rolf Schneider}
\date{}
\usepackage{graphicx}
\usepackage{epsfig}
\usepackage{pstricks}

\newcommand{\x}{{\bf x}}
\newcommand{\ur}{\mbox{\boldmath$u$}}

\newtheorem{theorem}{Theorem}
\newtheorem{lemma}{Lemma}[section]

\begin{document}
\maketitle

\begin{abstract}
A convexity point of a convex body is a point with the property that the union of the body and its reflection in the point is convex. It is proved that in the plane a typical convex body (in the sense of Baire category) has infinitely many convexity points. The proof makes use of the `middle hedgehog' of a planar convex body $K$, which is the curve formed by the midpoints of all affine diameters of $K$. The stated result follows from the fact that for a typical planar convex body the convex hull of the middle hedgehog has infinitely many exposed points.
\end{abstract}

\section{Introduction}\label{sec1}

The following question was posed to me by Shiri Artstein--Avidan: `Does every convex body $K$ in the plane have a point $z$ such that the union of $K$ and its reflection in $z$ is convex?' After some surprise about never having come across this simple question, and after some fruitless attempts to find counterexamples, this finally led to the following answer ({\cite{Sch16}). Here we call the point $z$ a {\em convexity point} of $K$ if $(K-z)\cup(z-K)$ is convex. 

\begin{theorem}\label{Thm1}
A convex body in the plane which is not centrally symmetric has three affinely independent convexity points.
\end{theorem}

A triangle and a Reuleaux triangle are examples of convex bodies with precisely three convexity points. This then raises the question whether the existence of just three convexity points is `typical'. We recall the meaning of this terminology. The space ${\mathcal K}^2$ of convex bodies in the plane with the Hausdorff metric is a complete metric space and hence a Baire space, that is, a topological space in which any intersection of countably many dense open sets is still dense. A subset of a Baire space is called {\em comeager} or {\em residual} if its complement is a {\em meager} set, that is, a countable union of nowhere dense sets (also said to be of {\em first Baire category}). The intersection of countable many comeager sets in a Baire space is still dense, which is a good reason to consider comeager sets as `large'. Therefore, one says that `most' convex bodies in the plane have a certain property, or that a `typical' planar convex body has this property, if the set of bodies with this property is comeager in ${\mathcal K}^2$. With this definition, we prove the following.

\begin{theorem}\label{Thm2}
A typical convex body in the plane has infinitely many convexity points.
\end{theorem}

A result from which this one follows will be formulated at the end of the next section, after some preparations.

For surveys on Baire category results in convexity, we refer the reader to Gruber \cite{Gru85, Gru93} and Zamfirescu \cite{Zam91, Zam09}.

\section{The middle hedgehog}\label{sec2}

We work in the Euclidean plane ${\mathbb R}^2$, with scalar product $\langle\cdot\,,\cdot\rangle$, induced norm $\|\cdot\|$ and unit circle ${\mathbb S}^1$. The set of convex bodies (nonempty, compact, convex subsets) in ${\mathbb R}^2$ is denoted by ${\mathcal K}^2$. We use the Hausdorff metric $\delta$, which is defined on all nonempty compact subsets of ${\mathbb R}^2$  (for notions from convex geometry not explained here, we refer to \cite{Sch14}). Let $K\in{\mathcal K}^2$ and $u\in{\mathbb S}^1$. By $H(K,u)$ we denote the supporting line of $K$ with outer unit normal vector $u$, and we call the line
$$ M_K(u):=\frac{1}{2}[H(K,u)+H(K,-u)]$$
the {\em middle line} of $K$ with normal vector $u$ (hence, $M_K(u)=M_K(-u)$). With $F(K,u):= K\cap H(K,u)$, which is the face of $K$ with outer normal vector $u$, we call the convex set
$$ Z_K(u):=\frac{1}{2}[F(K,u)+F(K,-u)]$$
(either a singleton or a segment) the {\em middle set} of $K$ with normal vector $u$. If $F(K,u)$ is one-pointed, we write $F(K,u)=\{x_K(u)\}$, and if also $F(K,-u)$ is one-pointed, then $Z_K(u)=\{m_K(u)\}$ with
$$ m_K(u)= \frac{1}{2}[x_K(u)+x_K(-u)].$$
We call $m_K(u)$ a {\em middle point} of $K$. The set
$$ {\mathcal M}_K:= \bigcup_{u\in{\mathbb S}^1} Z_K(u)$$
is the {\em middle hedgehog} of $K$. It is a closed curve, the locus of all midpoints of affine diameters, that is, chords of $K$ connecting pairs of boundary points lying in distinct parallel support lines. 
 
The following lemma, proved in \cite{Sch16}, was crucial for the proof of Theorem \ref{Thm1}.

\begin{lemma}\label{Lem1}
Suppose that $K\in{\mathcal K}^2$ has no pair of parallel edges. Then each exposed point of the convex hull of the middle hedgehog ${\mathcal M}_K$ is a convexity point of $K$.
\end{lemma}

We consider special examples of middle hedgehogs. First, let $K$ be a convex polygon with no pair of parallel edges. For each edge $F(K,u)$ of $K$ we have $F(K,-u)=\{x_K(-u)\}$ and $Z_K(u)= (1/2)[F(K,u)+x_K(-u)]$. (Each middle point belongs to some $Z_K(u)$ with suitable $u$.) The union ${\mathcal M}_K$ of these segments, over all unit normal vectors of the edges, is a closed polygonal curve.

Second, let $K\in{\mathcal K}^2$ be strictly convex. Then the support function of $K$, which we denote by $h(K,\cdot)$, is differentiable on ${\mathbb R}^2\setminus \{0\}$. To obtain a parametrization of ${\mathcal M}_K$, we choose an orthonormal basis $(e_1,e_2)$ of ${\mathbb R}^2$ and write 
$$\ur(\varphi):= (\cos\varphi)e_1+(\sin\varphi)e_2,\quad \varphi\in{\mathbb R};$$
then $(\ur(\varphi),\ur'(\varphi))$ is an orthonormal frame with the same orientation as $(e_1,e_2)$. We define
$$  \x(\varphi) := m_K(\ur(\varphi))$$
and
\begin{equation}\label{2.1} 
p(\varphi) := \frac{1}{2}\left[h_K(\ur(\varphi))-h_K(-\ur(\varphi))\right]
\end{equation}
for $\varphi\in[0,\pi]$. Note that $\x$ is a parametrized closed curve, since $m_K(u)=m_K(-u)$ for $u\in{\mathbb S}^1$. Since $h_K(u)=\langle x_K(u),u\rangle$, we have
\begin{equation}\label{2.2} 
p(\varphi)=\langle \x(\varphi), \ur(\varphi)\rangle.
\end{equation}
Differentiating (\ref{2.2}) and using that $x_K(u)= \nabla h_K(u)$ (where $\nabla$ denotes the gradient; see \cite{Sch14}, Corollary 1.7.3), we obtain
\begin{equation}\label{2.3}  
p'(\varphi) =\langle \x(\varphi),\ur'(\varphi)\rangle.
\end{equation}
The equations equations (\ref{2.2}) and (\ref{2.3}) together yield
$$ \x(\varphi) = p(\varphi) \ur(\varphi) +p'(\varphi)\ur'(\varphi),\quad \varphi\in[0,\pi].$$
This is a convenient parametrization of the middle hedgehog. The intersection point of the middle lines $M_K(\ur(\varphi))$ and $M_K(\ur(\varphi+\varepsilon))$ converges to $m_K(\ur(\varphi))$ for $\varepsilon\to 0$, thus
$\x$ is the envelope of the family of middle lines of $K$, suitably parametrized. We remark that generalized envelopes of more general line families were studied in \cite{HS53}.

We remark further that in the terminology of Martinez--Maure (see \cite{MM95, MM97}, for example, also \cite{MM99, MM06}), the curve $\x$ is a planar `projective hedgehog'. The set $\{\x(\varphi):\varphi\in[0,\pi)\}$ has been introduced and investigated as the `midpoint parallel tangent locus' in \cite{Hol01} and has been named the `area evolute' in \cite{Gib08}; a further study appears in \cite{Cra14}.

According to Lemma \ref{Lem1} and the fact that a typical convex body is strictly convex, Theorem \ref{Thm2} is a consequence of the following result.

\begin{theorem}\label{Thm3}
For a typical convex body in the plane, the convex hull of the middle hedgehog has infinitely many exposed points.\end{theorem}

\section{Proof of Theorem \ref{Thm3}}\label{sec3}

By ${\mathcal K}^2_*$ we denote the set of strictly convex convex bodies in ${\mathcal K}^2$. The set ${\mathcal K}^2_*$ is a dense $G_\delta$ set in ${\mathcal K}^2$ and hence is also a Baire space. Every set that is comeager in ${\mathcal K}^2_*$ is also comeager in ${\mathcal K}^2$. 

To begin with the proof of Theorem \ref{Thm3}, we set
$$ {\mathcal A}:= \{K\in{\mathcal K}^2_*: {\rm conv}{\mathcal M}_K \mbox{ has only finitely many exposed points}\}$$
and, for $k\in {\mathbb N}$, 
$$ {\mathcal A}_k:= \{K\in{\mathcal K}^2_*: {\rm conv}{\mathcal M}_K \mbox{ has at most $k$ exposed points}\}.$$
We shall prove the following facts.

\begin{lemma}\label{Lem2}
Each set ${\mathcal A}_k$ is closed in ${\mathcal K}^2_*$.
\end{lemma}

\begin{lemma}\label{Lem3}
Each set ${\mathcal A}_k$ is nowhere dense in ${\mathcal K}^2_*$.
\end{lemma}

When this has been proved, then we know that the set ${\mathcal A}= \bigcup_{k\in{\mathbb N}} {\mathcal A}_k$ is meager. Hence its complement, which is the set of all $K\in{\mathcal K}^2_*$ for which ${\rm conv}{\mathcal M}_K$ has infinitely many exposed points, is comeager in ${\mathcal K}^2_*$ and hence in ${\mathcal K}^2$. This is the assertion of Theorem \ref{Thm3}.

\vspace{2mm}

\noindent{\em Proof of Lemma} \ref{Lem2}. First we show that on ${\mathcal K}^2_*$, the mapping $K\mapsto{\mathcal M}_K$ is continuous (this would not be true if ${\mathcal K}^2_*$ were replaced by ${\mathcal K}^2$).

Let $(K_i)_{i\in{\mathbb N}}$ be a sequence in ${\mathcal K}^2_*$ converging to some $K\in {\mathcal K}^2_*$. To show that ${\mathcal M}_{K_i}\to {\mathcal M}_K$ in the Hausdorff metric for $i\to\infty$, we use Theorem 1.8.8 of \cite{Sch14} (it is formulated for convex bodies, but as its proof shows, it holds for connected compact sets---or see Theorems 12.2.2 and 12.3.4 in \cite{SW08}). 

Let $x\in{\mathcal M}_{K}$. Then there is a vector $u\in{\mathbb S}^1$ with $x=(1/2)[x_K(u)+x_K(-u)]$. The sequence $(x_{K_i}(u))_{i\in{\mathbb N}}$ has a convergent subseqence, and its limit is a boundary point of $K$ with outer normal vector $u$, hence equal to $x_K(u)$. Since this holds for every convergent subsequence, the sequence $(x_{K_i}(u))_{i\in{\mathbb N}}$ itself converges to $x_K(u)$. Similarly, the sequence $(x_{K_i}(-u))_{i\in{\mathbb N}}$ converges to $x_K(-u)$. It follows that $m_{K_i}(u) = (1/2)[x_{K_i}(u)+x_{K_i}(-u)]\to (1/2)[x_{K}(u)+x_{K}(-u)]=x$ for $i\to\infty$, and here $m_{K_i}(u)\in {\mathcal M}_{K_i}$. Thus, each point in ${\mathcal M}_K$ is the limit of a sequence $(m_i)_{\in{\mathbb N}}$ with $m_i\in {\mathcal M}_{K_i}$ for $i\in{\mathbb N}$.

Let $x_{i(j)}\in{\mathcal M}_{K_{i(j)}}$ for a subsequence $(i(j))_{j\in{\mathbb N}}$, and suppose that $x_{i(j)}\to x$ for $j\to\infty$. Then $x_{i(j)}=(1/2)[x_{K_{i(j)}}(u_j)+x_{K_{i(j)}}(-u_j)]$ for suitable $u_j\in{\mathbb S}^1$ ($j\in{\mathbb N}$). There is a convergent subsequence of $(u_j)_{j\in{\mathbb N}}$, and we can assume that this is the sequence $(u_j)_{j\in{\mathbb N}}$ itself, say $u_j\to u$ for $j\to\infty$. Then $x_{K_{i(j)}}(u_j)\to x_K(u)$ and $x_{K_{i(j)}}(-u_j)\to x_K(-u)$, hence $x_{i(j)}\to (1/2)[x_K(u)+x_K(-u)]\in {\mathcal M}_K$. It follows that $x\in{\mathcal M}_K$. This completes the continuity proof for the mapping $K\mapsto{\mathcal M}_K$.

To show that ${\mathcal A}_k$ is closed in ${\mathcal K}^2_*$, let $(K_i)_{i\in{\mathbb N}}$ be a sequence in ${\mathcal A}_k$ converging to some $K\in {\mathcal K}^2_*$. As just shown, we have ${\mathcal M}_{K_i}\to {\mathcal M}_K$ and hence also ${\rm conv}{\mathcal M}_{K_i}\to {\rm conv}{\mathcal M}_K$ for $i\to\infty$, since the convex hull mapping is continuous (even Lipschitz, see \cite{Sch14}, p. 64). Since each ${\rm conv}{\mathcal M}_{K_i}$ is a convex polygon with at most $k$ vertices, also ${\rm conv}{\mathcal M}_{K}$ is a convex polygon with at most $k$ vertices, thus $K\in {\mathcal A}_k$. This completes the proof of Lemma \ref{Lem2}. \qed

To prepare the proof of Lemma \ref{Lem3}, we need to have a closer look at the middle hedgehog ${\mathcal M}_P$ of a convex polygon $P$. We assume in the following that $P$ has interior points and has no pair of parallel edges.

First, the unoriented normal directions of the edges of $P$ have a natural cyclic order. We may assume, without loss of generality, that no edge of $P$ is parallel to the basis vector $e_1$. Then there are angles $-\pi/2 < \varphi_1 < \varphi_2<\dots< \varphi_k< \pi/2$ such that, for each $i \in \{1,\dots,k\}$, either $\ur(\varphi_i)$ or $-\ur(\varphi_i)$ is an outer normal vector of an edge of $P$ (not both, since $P$ does not have a pair of parallel edges), and all unit normal vectors of the edges of $P$ are obtained in this way. We denote by $E_i$ the edge of $P$ that is orthogonal to $\ur(\varphi_i)$. We call the pair $(E_i, E_{i+1})$ {\em consecutive} (where $E_{k+1} := E_1$; this convention is also followed below), and in addition we call it {\em adjacent} if $E_i\cap E_{i+1}$ is a vertex of $P$. For an angle $\psi\in[-\pi/2,\pi/2)$ we say that $\psi$ is {\em between} $\varphi_i$ and $\varphi_{i+1}$ if either $i\in \{1,\dots,k-1\}$ and $\varphi_i<\psi<\varphi_{i+1}$, or $i=k$ and either $-\pi/2 <\psi<\varphi_1$ or $\varphi_k<\psi\le \pi/2$. Let $(E_i, E_{i+1})$ be a consecutive pair. The following facts, to be used below, follow immediately from the definitions. If $\psi$ is between $\varphi_i$ and $\varphi_{i+1}$, then $\ur(\psi)$ is not a normal vector of an edge of $P$. Suppose that, say, $\ur(\varphi_i)$ is the outer normal vector of $E_i$. If $(E_i,E_{i+1})$ is adjacent, then $\ur(\varphi_{i+1})$ is the outer normal vector of $E_{i+1}$. If $(E_i,E_{i+1})$ is not adjacent, then $\ur(\varphi_{i+1})$ is the inner normal vector of $E_{i+1}$.  These definitions of $E_i$ and $\varphi_i$ will be used in the rest of this note.

Now let $p$ and $q$ be {\em opposite} vertices of $P$, that is, vertices with $H(P,\ur(\psi))\cap P=\{p\}$ and $H(P,-\ur(\psi))\cap P=\{q\}$ for some $\psi$. After interchanging $p$ and $q$, if necessary, we can assume that $\psi\in[-\pi/2,\pi/2)$. Then there is a unique index $i\in\{1,\dots,k\}$ such that $\psi$ is between $\varphi_i$ and $\varphi_{i+1}$. The middle sets $Z_P(\ur(\varphi_i))$ and $Z_P(\ur(\varphi_{i+1}))$ have the midpoint $x=(p+q)/2$ in common. We say that $x$ is a {\em weak corner} of the middle hedgehog ${\mathcal M}_P$ if the pair $(E_i,E_{i+1})$ is adjacent, and $x$ is a {\em strong corner} of ${\mathcal M}_P$ if $(E_i,E_{i+1})$ is not adjacent. If $x$ is a weak corner, then the middle sets $Z_P(\ur(\varphi_i))$ and $Z_P(\ur(\varphi_{i+1}))$ lie on different sides of the line through $p$ and $q$, and if $x$ is a strong corner, then $Z_P(\ur(\varphi_i))$ and $Z_P(\ur(\varphi_{i+1}))$ lie on the same side of this line.

\vspace{-1.2cm}

\begin{center}
\resizebox{16cm}{!}{
\setlength{\unitlength}{1cm}
\begin{pspicture}(0,0)(14,11)

\psline[linewidth=1pt]{-}(6.8,0.5)(2.54,1.4) 
\psline[linewidth=1pt]{-}(2.54,1.4)(1.04,4.62) 
\psline[linewidth=1pt]{-}(1.04,4.62)(1.8,7.4) 
\psline[linewidth=1pt]{-}(1.8,7.4)(8.24,10) 
\psline[linewidth=1pt]{-}(8.24,10)(12.9,6.6) 
\psline[linewidth=1pt]{-}(12.9,6.6)(12.7,4.3) 
\psline[linewidth=1pt]{-}(12.7,4.3)(10.66,1.24) 
\psline[linewidth=1pt]{-}(10.66,1.24)(6.8,0.5) 

\psline[linewidth=0.4pt](6.8,0.5)(8.24,10)
\psline[linewidth=0.4pt](8.24,10)(2.54,1.4)
\psline[linewidth=0.4pt](8.24,10)(10.66,1.24)
\psline[linewidth=0.4pt](1.8,7.4)(12.7,4.3)
\psline[linewidth=0.4pt](1.8,7.4)(10.66,1.24)
\psline[linewidth=0.4pt](12.9,6.6)(1.04,4.62)
\psline[linewidth=0.4pt](12.9,6.6)(2.54,1.4)
\psline[linewidth=0.4pt](1.04,4.62)(12.7,4.3)

\psline{-}(5.39,5.7)(7.72,4)
\psline{-}(7.72,4)(6.97,5.61)
\psline{-}(6.97,5.61)(6.87,4.46)
\psline{-}(6.87,4.46)(7.25,5.83)
\psline{-}(7.25,5.83)(6.23,4.32)
\psline{-}(6.23,4.32)(9.45,5.62)
\psline{-}(9.45,5.62)(7.52,5.25)
\psline{-}(7.52,5.25)(5.39,5.7)

\rput(8.8,0.5){$E_1$}
\rput(4.6,8.95){$E_2$}
\rput(12,2.7){$E_3$}
\rput(1.05,6.1){$E_4$}
\rput(13.1,5.5){$E_5$}
\rput(1.5,2.9){$E_6$}
\rput(10.7,8.6){$E_7$}
\rput(4.4,0.65){$E_8$}

\end{pspicture}
}
\end{center}

\vspace{-2mm}

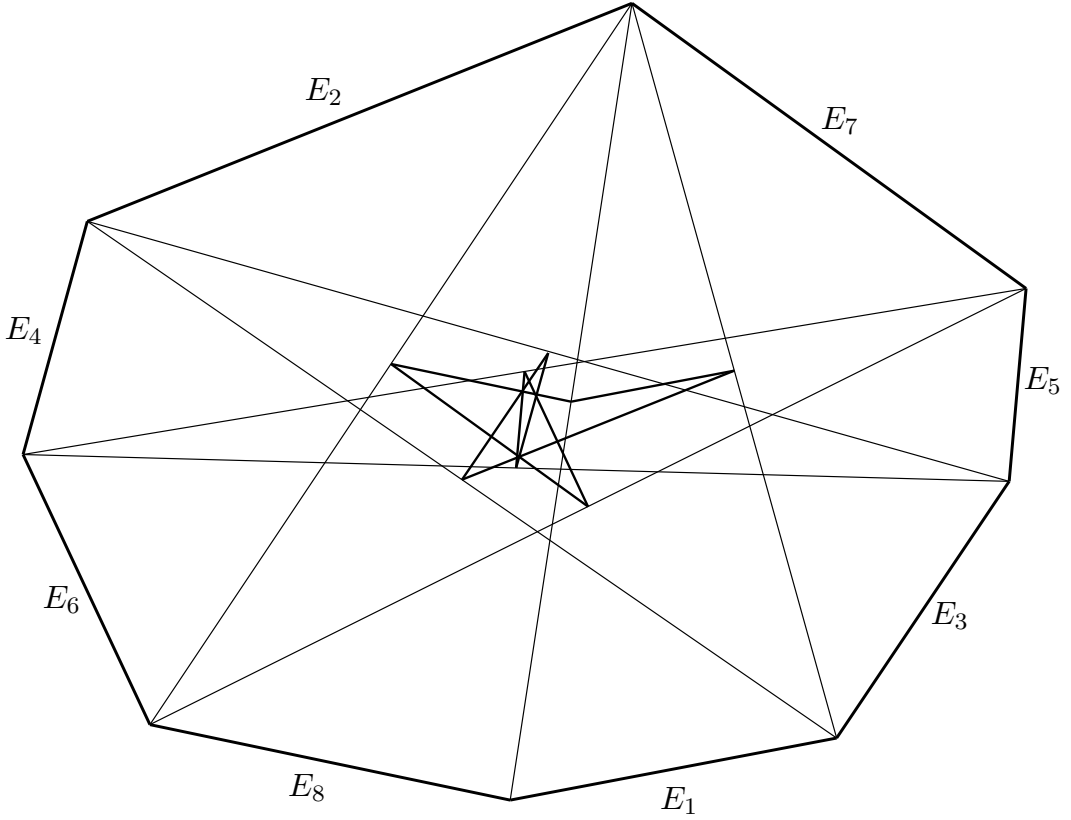
\noindent Figure 1: The middle hedgehog has one weak corner and seven strong corners, five of which are vertices of the convex hull.

\vspace{5mm}

\begin{lemma}\label{Lem4}
A weak corner of the middle hedgehog ${\mathcal M}_P$ is not a vertex of ${\rm conv}{\mathcal M}_P$.
\end{lemma}

\begin{proof}
We begin with an arbitrary vertex $x$ of ${\rm conv}{\mathcal M}_P$. Since ${\mathcal M}_P$ is the union of the finitely many middle sets $Z_P(\ur(\varphi_i))$ (with $\varphi_i$ as above), the point $x$ must be one of the endpoints of these segments, thus $x$ is either a weak or a strong corner of ${\mathcal M}_P$. 

We need to recall some facts from the proof of Lemma 6 in \cite{Sch16}. As there, we may assume, without loss of generality (after applying a rigid motion to $P$), that $x=0$ and that the orthonormal basis $(e_1,e_2)$ of ${\mathbb R}^2$  is such that
\begin{equation}\label{3.1}
\langle y,e_2\rangle >0 \quad \mbox{for each } y\in {\rm conv}{\mathcal M}_P \setminus\{0\}.
\end{equation}
Let $L$ be the line through $0$ that is spanned by $e_1$. For $\varphi\in (-\pi/2,\pi/2)$, the middle line $M_P(\ur(\varphi))$ intersects the line $L$ in a point which we write as $f(\varphi)e_1$, thus defining a continuous function $f:(-\pi/2,\pi/2)\to{\mathbb R}$. It was shown in \cite{Sch16} that 
$$ f(\varphi) =\frac{p(\varphi)}{\cos\varphi}.$$
At almost all $\varphi$, the functions $\varphi\mapsto h(P,\ur(\varphi))$ and $\varphi\mapsto h(P,-\ur(\varphi))$ are differentiable, hence the same holds for the function $f$, and where this holds, we have
\begin{equation}\label{3.2}
f'(\varphi)= \frac{\langle m_P(\ur(\varphi)),e_2\rangle}{\cos^2\varphi},
\end{equation}
as shown in \cite{Sch16}.

We now first recall the rest of the proof of Lemma 6 in \cite{Sch16}, in a slightly simplified version. The claim to be proved is that
\begin{equation}\label{3.3}
0 \in M_P(\ur(\varphi))  \mbox{ for some }\varphi\in(-\pi/2,\pi/2)\quad  \Longrightarrow \quad 0\in Z_P(\ur(\varphi)).
\end{equation}

By (\ref{3.2}) and (\ref{3.1}) we have $f'(\varphi)\ge 0$ for almost every $\varphi\in(-\pi/2,\pi/2)$. We conclude that the function $f$ (which is locally Lipschitz and hence the integral of its derivative) is weakly increasing on $(-\pi/2, \pi/2)$. Therefore, the set $I:=\{\varphi\in(-\pi/2, \pi/2): f(\varphi)=0\}$ is a closed interval (possibly one-pointed). Since $0\in{\mathcal M}_P$, there is some $\varphi_0\in (-\pi/2,\pi/2)$ with $0\in Z_P(\ur(\varphi_0))$. If $I$ is one-pointed, then $I=\{\varphi_0\}$, and $0\notin M_P(\ur(\varphi))$ for $\varphi\not=\varphi_0$. Thus, (\ref{3.3}) holds in this case. If $I$ is not one-pointed, then $f'(\varphi)=0$ for $\varphi \in {\rm relint}\,I$ and hence, by (\ref{3.2}) and (\ref{3.1}), $m_P(\ur(\varphi))=0$ for $\varphi \in {\rm relint}\,I$. By continuity, we have $0\in Z_P(\ur(\varphi))$ for all $\varphi\in I$. This shows that (\ref{3.3}) holds generally.

Now we can finish the proof of Lemma \ref{Lem4}. Suppose, to the contrary, that $0$ is a weak corner of ${\mathcal M}_P$. Then there is a consecutive, adjacent pair $(E_i,E_{i+1})$ of edges of $P$ such that $E_i\cap E_{i+1}=\{p\}$ for a vertex $p$ of $P$ and the line $H(P,\ur(-\pi/2))$ supports $P$ at $p$. This is only possible if $(E_i,E_{i+1})= (E_k,E_{k+1})$. In this case, all the middle lines $M_P(\psi)$ with $\psi$ between $\varphi_k$ and $\varphi_{k+1}=\varphi_1$ pass through $0$. This means that the function $f$ defined above satisfies $f(\varphi)=0$ for $-\pi/2<\varphi \le \varphi_1$ and for $\varphi_k\le\varphi<\pi/2$. But since $f$ is increasing, it must then vanish identically, which is a contradiction, since $P$ is not centrally symmetric. This contradiction completes the proof of Lemma \ref{Lem4}.
\end{proof}

\vspace{2mm}

\noindent{\em Proof of Lemma} \ref{Lem3}. Let $k\in{\mathbb N}$. Since ${\mathcal A}_k$ is closed by Lemma \ref{Lem2}, the proof that ${\mathcal A}_k$ is nowhere dense amounts to showing that ${\mathcal A}_k$ has empty interior in ${\mathcal K}^2_*$. For this, let $K\in {\mathcal A}_k$ and $\varepsilon>0$ be given. We show that the $\varepsilon$-neighborhood of $K$ contains an element of ${\mathcal K}^2_*\setminus {\mathcal A}_k$. 

In a first step, we choose a convex polygon $P$ with
\begin{equation}\label{3.0} 
K\subset {\rm int} P,\quad P \subset {\rm int}(K+\varepsilon B^2),
\end{equation}
where $B^2$ denotes the closed unit disc of ${\mathbb R}^2$. We can do this in such a way that $P$ satisfies the following assumptions. First, $P$ has no pair of parallel edges. Second, $P$ has no `long' edge, by which we mean an edge the endpoints of which are opposite points of $P$. The goal of the following is to perform small changes on the polygon $P$ so that the number of vertices of ${\rm  conv}{\mathcal M}_P$ is increased.

Let $x$ be a vertex of ${\rm conv}{\mathcal M}_P$. It is a corner of ${\mathcal M}_P$, and by Lemma \ref{Lem4} a strong corner. Therefore, there is a consecutive, non-adjacent edge pair $(E_i,E_{i+1})$ of $P$ and there are an endpoint $p$ of $E_i$ and an endpoint $q$ of $E_{i+1}$ such that $x=(p+q)/2$. 

We position $P$ and choose the orthonormal basis $(e_1,e_2)$ in such a way that $x=0$, that $e_1$ is a positive multiple of $q$, and that $\langle y,e_2\rangle\ge 0$ for all $y\in E_i\cup E_{i+1}$ (note that $E_i$ and $E_{i+1}$ lie on the same side of the line through $p$ and $q$, since $0$ is a strong corner of ${\mathcal M}_P$).

We may assume (the other case is treated similarly) that $\ur(\varphi_i)$ is the inner normal vector of $E_i$; then $\ur(\varphi_{i+1})$ is the outer normal vector of $E_{i+1}$. Let $E_j\not= E_i$ be the other edge of $P$ with endpoint $p$, and let $E_m\not= E_{i+1}$ be the other edge of $P$ with endpoint $q$. The edges $E_j$ and $E_m$ do not lie in the line through $p$ and $q$, since $P$ has no long edge. We have $\varphi_m<\varphi_i< \varphi_{i+1} < \varphi_j$, since $\ur(\psi)$ with $\psi$ between $\varphi_i$ and $\varphi_{i+1}$ is not a normal vector of an edge of $P$. 

\vspace{-1.2cm}

\begin{center}
\resizebox{15cm}{!}{
\setlength{\unitlength}{1cm}
\begin{pspicture}(0,0)(14,11)

\psline[linewidth=1pt]{*-*}(1.3,3.64)(12,3.64) 
\psline[linewidth=1pt]{-}(7.6,0.75)(12,3.64) 
\psline[linewidth=1pt]{-}(5,0.75)(1.3,3.64) 
\psline[linewidth=0.5pt]{*-}(6.65,3.64)(5.4,6.5) 
\psline[linewidth=0.5pt]{-}(6.65,3.64)(7.7,6.5) 
\psline[linestyle=dashed]{-}(4.5,2.7)(8.8,4.59) 
\psline[linewidth=1pt]{-}(1.3,3.64)(3.5,9.2) 
\psline[linewidth=1pt]{*-*}(2.78,2.5)(2.95,7.81) 
\psline[linewidth=1pt]{*-}(9.199,1.8)(10.4,8.9) 
\psline[linewidth=1pt]{-}(9.14,5.3)(12.7,6.94) 
\psline[linewidth=1pt]{-}(12,3.64)(9.6,9.36) 
\psline[linewidth=1pt]{-}(9.1,8.73)(12.7,6) 
\psline{*-}(10.22,7.87)(10.22,7.87)
\psline[linewidth=1pt]{-*}(9.199,1.8)(10.702,6.734) 
\psline{*-}(11.91,6.59)(11.91,6.59) 
\psline{*-}(10.45,5.9)(10.45,5.9)

\rput(3.3,9.6){$E_i$}
\rput(10,9.6){$E_{i+1}$}
\rput(4.1,0.8){$E_j$}
\rput(8.4,0.8){$E_m$}
\rput(0.9,3.6){$p$}
\rput(12.3,3.6){$q$}
\rput(2.2,7.9){$p+t_2$}
\rput(11,7.9){$q+s_2$}
\rput(2,2.3){$p+t_1$}
\rput(9.9,1.7){$q+s_1$}
\rput(10.8,6.9){$q+s$}
\rput(13.3,6.5){$q+s_2+t_1$}
\rput(11.42,5.7){$q+s_1+t_2$}
\rput(6.65,3.3){$0$}
\rput(8,4.6){$S$}
\rput(10.4,4.6){$L_q$}

\end{pspicture}
}
\end{center}

\vspace{-8mm}

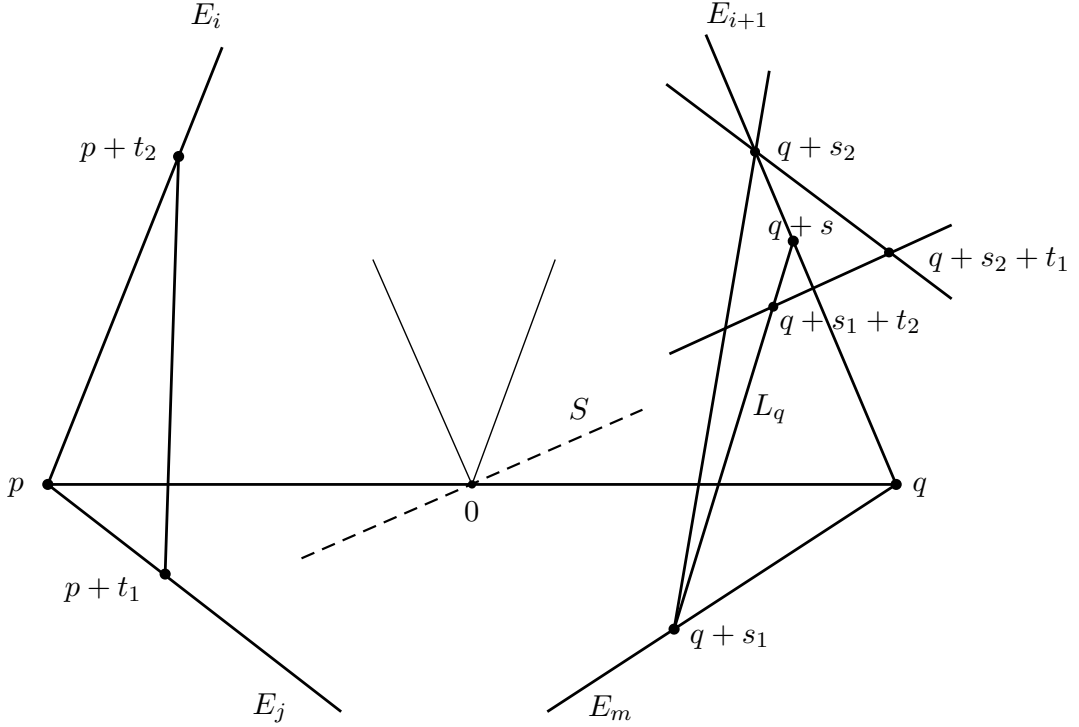
\noindent Figure 2: The vertices $p$ and $q$ are cut off by new edges, in the figure with endpoints $p+t_1, p+t_2$, respectively $q+s_1,q+s_2$.

\vspace{5mm}


By assumption, $0$ is a vertex of ${\rm conv}{\mathcal M}_P$. Therefore, there is a support line $S$ of ${\rm conv}{\mathcal M}_P$ which has intersection $\{0\}$ with ${\rm conv}{\mathcal M}_P$. Since $S$ supports also the convex hull of $Z_P(\ur(\varphi_i))$ and $Z_P(\ur(\varphi_{i+1}))$, the (with respect to ${\rm conv}{\mathcal M}_P$) outer unit normal vector $\ur(\alpha)$ of the support line $S$ has an angle $\alpha$ that satisfies either $-\pi/2 \le \alpha <\varphi_i$ or $\varphi_{i+1}-\pi/2<\alpha < -\pi/2$. We assume that  $-\pi/2 \le \alpha <\varphi_i$; the other case is treated analogously, with the roles of $p,E_i,E_j$ and $q,E_{i+1},E_m$ interchanged.

In the following, $t_1$ and $t_2$ denote vectors such that $p+t_1\in E_j$ and $p+t_2\in E_i$. For such vectors, let $\psi_p=\psi_p(t_1,t_2)$ with $\varphi_i < \psi_p < \varphi_{i+1}$ be the angle for which $\ur(\psi_p)$ is orthogonal to the line through $p+t_1$ and $p+t_2$. Trivially, there are a constant $c>0$ and a continuous function $\gamma: [\varphi_i, \varphi_{i+1}] \to {\mathbb R}^+$ with $\lim_{\psi\to \varphi_i}\gamma(\psi) =0$ such that
\begin{align}\label{3.6}  
\|t_1\| <c\|t_2\|\quad &\Longrightarrow  \quad \psi_p(t_1,t_2) <\varphi_{i+1},\\
\label{3.7}
\psi\in [\varphi_i, \varphi_{i+1}]\mbox{ and } \|t_1\| >\gamma(\psi) \|t_2\|\quad  & \Longrightarrow  \quad  \psi_p(t_1,t_2)> \psi. 
\end{align}

Let $L_q$ be a line parallel to $E_i$ and strongly separating $q$ from the other endpoints of $E_{i+1}$ and $E_m$. This line intersects $E_m$ in a point $q+s_1$, and it intersects $E_{i+1}$ in a point $q+s$. We choose the line $L_q$ so close to $q$ that the vector $t:= s-s_1$ satisfies $p+t\in E_i$. 

Let $0<\tau<1$, and let $\sigma>1$ be such that $q+\sigma s\in E_{i+1}$. The line through the point $q+s_1+\tau t$ parallel to the support line $S$ and the line through $q+\sigma s$ parallel to $E_j$ intersect in a point $q+\sigma s+t_1$. This defines a vector function $t_1=t_1(\tau,\sigma)$, with the property that $\|t_1\|$ is strictly increasing in $\sigma$. For $\tau,\sigma\to 1$ we have $\|t_1\|\to 0$; in particular, $p+t_1\in E_j$ if $\tau,\sigma$ are sufficiently close to $1$. Therefore, we can fix $t_2=\tau t $ (so that $t_1$ now depends only on $\sigma$) and choose $\sigma_0>1$ such that 
$$ p+t_1(\sigma)\in E_j\quad\mbox{and}\quad\|t_1(\sigma)\|<c\|t_2\|\quad\mbox{for } 1<\sigma\le \sigma_0.$$ 
Let $\psi_q(\sigma)\in(\varphi_i,\varphi_{i+1})$ be the angle for which $\ur(\psi_q)$ is orthogonal to the line through $q+s_1$ and $q+\sigma s$. We choose $\sigma_1$ with $1<\sigma_1<\sigma_0$ so close to $1$ that
$$ \gamma(\psi_q(\sigma_1)) < \|t_1(\sigma_1)\|/\|t_2\|,$$
which is possible because of $\lim_{\psi\to \varphi_i}\gamma(\psi) =0$ and $\lim_{\sigma\to 1}\|t_1(\sigma)\|>0$. For $\sigma\in (\sigma_1,\sigma_0]$ sufficiently close to $\sigma_1$ we then have
$$ \gamma(\psi_q(\sigma)) < \|t_1(\sigma)\|/\|t_2\| \le \|t_1(\sigma_0)\|/\|t_2\| <c.$$
Therefore, by (\ref{3.6}) and (\ref{3.7}), the angles $\psi_q=\psi_q(\sigma)$ and $\psi_p=\psi_p(t_1(\sigma), t_2)$ satisfy
\begin{equation}\label{3.9}
\varphi_i <\psi_q <\psi_p<\varphi_{i+1}.
\end{equation}
In the following we write $t_1(\sigma)=t_1$ and $\sigma s= s_2$.

Now we choose a number $0<\lambda<1$ and replace $p+t_1, p+t_2, q+s_1, q+s_2$ respectively by $p+\lambda t_1, p+ \lambda t_2, q+\lambda s_1, q+\lambda s_2$. This does not change the angles $\psi_p,\psi_q$. We replace $P$ by the polygon $P_\lambda$ that is the convex hull of the points $p+\lambda t_1, p+ \lambda t_2, q+\lambda s_1, q+\lambda s_2$ and of the vertices of $P$ different from $p$ and $q$. By choosing $\lambda$ sufficiently small, we can achieve that still
$$ K\subset {\rm int} P_\lambda.$$
Note that $P_\lambda \subset {\rm int}(K+\varepsilon B^2)$ holds trivially. 

By decreasing $\lambda$ further, if necessary, we can also achieve that ${\rm conv}{\mathcal M}_{P_\lambda}$ has more vertices than ${\rm conv}{\mathcal M}_P$, as we now show. First we notice that the inequalities (\ref{3.9}) imply that $p+\lambda t_2$ and $q+\lambda s_1$ are opposite vertices of $P_\lambda$ and that also $p+\lambda t_1$ and $q+\lambda s_2$ are opposite vertices of $P_\lambda$. Hence,
$$\frac{1}{2}(p+\lambda t_2 + q+\lambda s_1) = \frac{\lambda}{2}(s_1+t_2)=:y_\lambda$$
and 
$$ \frac{1}{2}(p+\lambda t_1 +q+\lambda s_2)= \frac{\lambda}{2}(s_2+t_1)=: z_\lambda$$ 
are strong corners of ${\mathcal M}_{P_\lambda}$. By construction,
\begin{equation}\label{3.4}
z_\lambda-y_\lambda = \frac{\lambda}{2}[(q+s_2+t_1)-(q+s_1+t_2)] \quad\mbox{is parallel to } S.
\end{equation}
Since the non-zero vectors $s_1-s_2$ and $t_1-t_2$ have different directions, we have $y_\lambda\not= z_\lambda$.

Let $v_0=0,v_1,\dots,v_r$ be the vertices of ${\rm conv}{\mathcal M}_P$. They are corner points of ${\mathcal M}_P$. To each $i\in\{0,\dots,r\}$ we choose a line $L_i$ that strongly separates $v_i$ from the other vertices; the particular line $L_0$ is chosen parallel to the support line $S$. We can choose a number $\eta>0$ such that, for any  $\bar v_0,\dots,\bar v_r \in {\mathbb R}^2$  with $\|\bar v_i- v_i\| < \eta$ for $i=0,\dots,r$, the line $L_i$ strongly separates $\bar v_i$ from the points $\bar v_j\not= \bar v_i$. Then we can further decrease $\lambda$ so that the $\eta$-neighborhood of each $v_i$, $i=1,\dots,r$, contains at least one corner point of ${\mathcal M}_{P_\lambda}$, and that the $\eta$-neighborhood of $0$ contains the points $y_\lambda$ and $z_\lambda$. Since $L_0$ is parallel to $S$, it follows from (\ref{3.4}) that $y_\lambda$ and $z_\lambda$ are both vertices of ${\rm conv}{\mathcal M}_{P_\lambda}$. Thus, ${\rm conv}{\mathcal M}_{P_\lambda}$ has more vertices than ${\rm conv}{\mathcal M}_P$.

Since (\ref{3.0}) with $P$ replaced by $P_\lambda$ still holds, we can repeat the procedure. After finitely many steps, we obtain a polygon $Q$ with
\begin{equation}\label{3.5} 
K\subset {\rm int}\, Q,\quad Q \subset {\rm int}(K+\varepsilon B^2)
\end{equation}
for which ${\mathcal M}_{Q}$ has more than $k$ vertices. Finally, we replace $Q$ by a strictly convex body $M$, by replacing each edge of $Q$ by a circular arc of large positive radius $R$. If $R$ is large enough, then (\ref{3.5}), with $Q$ replaced by $M$, still holds, and the number of vertices of ${\rm conv}{\mathcal M}_M$ is the same as for ${\rm conv}{\mathcal M}_{Q}$. Thus, in the $\varepsilon$-neighborhood of $K$ we have found an element of ${\mathcal K}^2_*\setminus {\mathcal A}_k$. \qed

\noindent Author's address:\\[2mm]
Rolf Schneider\\
Mathematisches Institut, Albert-Ludwigs-Universit{\"a}t\\
D-79104 Freiburg i. Br., Germany\\
E-mail: rolf.schneider@math.uni-freiburg.de

\end{document}